\documentclass[12pt]{article}
\usepackage[english]{babel}
\usepackage{times,amsfonts,amsmath,amsxtra,amsthm,amssymb,latexsym}
\usepackage[cp866]{inputenc}
\usepackage[dvips]{epsfig}
\usepackage{graphicx}

\advance\voffset -0.6in

\makeatletter\@addtoreset {equation}{section}\makeatother
\theoremstyle{plain}

\newtheorem{theorem}{Theorem}[section]
\newtheorem{corollary}[theorem]{Corollary}

\newtheorem{proposition}[theorem]{Proposition}

\begin{document}

\title{\bf Sudoku Symmetry Group}

\author{Vasiliy Osipov \\ {\small Department of
Mathematics, Far Eastern National University, Russia}\\ {\small
e-mail: vosipov@ext.dvgu.ru}}

\date{\today}
\maketitle

\begin{abstract}
The mathematical aspects of the popular logic game Sudoku
incorporate a significant number of the group theory concepts. In
this note, we describe all symmetric transformations of the Sudoku
grid. We do not intend to obtain a new strategy of solving Sudoku
and do not describe basic ideas of the game which can be found in
numerous other sources.
\end{abstract}

\section{Symmetric transformations of the Sudoku grid}
To define a Sudoku grid we will use the usual matrix notations, i.e:
\begin{enumerate}
    \item  $a_{i,j}$ - this is the grid cell with the row index $i$ and the
column index $j$.
    \item  The Sudoku $3 \times 3$ squares we will call blocks.
    \item  Three horizontal sectors (bands) are defined by the following sets of
rows: \linebreak \{ 1, 2, 3 \}, \{ 4, 5, 6 \} and \{ 7, 8, 9 \}.
    \item Three vertical sectors (stacks) are defined by the following sets of
columns: \linebreak \{ 1, 2, 3 \}, \{ 4, 5, 6 \} and \{ 7, 8, 9 \}.
\end{enumerate}

Let us fix some completely filled Sudoku grid. The natural question
arises: How many new Sudoku grids can be obtained from the fixed one
by different symmetric transformations? We will give a complete
description of all symmetric transformations and will show that any
symmetric transformation $S$ of the Sudoku grid can be represented
as some permutation of the rows in the horizontal sectors,
permutation of the horizontal sectors itself and matrix
transposition. We will prove that all other symmetric
transformations that are described by many authors  can be obtained
by some combination of the main three mentioned above.

We should introduce some notations:
\begin{enumerate}
    \item $A_1 = \{E\}$ where $E$ - is the identical transformation.
    \item $A_2 = \{d \}$ where $d$ - is the matrix transposition.
    \item $A_3 = \{r_i\}$ where index $i = \overline{1, 6^4 - 1}$ and  $r_i \neq E.$
\end{enumerate}

For any element from $A_3$ at least one of the following conditions
holds true: $r_i^2 = E$, $r_i^3 = E$, $r_i^6 = E$, $r_i^9 = E$,
$r_i^12 = E.$ Where $r_i$ is the permutation of the rows in some
horizontal sector, permutation of sectors itself or both. Each $r_i$
is related to some substitution of the following type:

$$  r \rightarrow \left(%
\begin{array}{ccc}
  1\, 2 \, 3 & 4 \,5 \,6 & 7\, 8 \,9 \\
  \alpha_1 \alpha_2 \alpha_3 & \beta_1 \beta_2 \beta_3 & \gamma_1 \gamma_2 \gamma_3 \\
\end{array}%
\right).$$

Where each triplet $(\alpha_1 \alpha_2 \alpha_3)$, $(\beta_1 \beta_2
\beta_3)$, $(\gamma_1 \gamma_2 \gamma_3)$ this is $(1, 2, 3)$ or
$(4, 5, 6)$ or $(7, 8, 9)$ i.e these are triplet's permutations. The
total number of them is $6$. The total number of the permutations
inside the triplets equals $6^3$. Therefore the total number of all
possible $r_i$ is equal to $(6^4 - 1)$ (we do not take into account
the identical one). These permutations define a subgroup in $S_9$
(if we add the identical one of course).

\begin{enumerate}
    \setcounter{enumi}{3}
    \item $A_4 = \{ r_k d\}$ where $r_k \neq E$ and $k = \overline{1, 6^4 - 1}.$
    \item $A_5 = \{ d r_l \}$ where  $r_l \neq E$ and $l = \overline{1, 6^4 - 1}.$
    \item $A_6 = \{ d r_j d \}$ where $r_j \neq E$ and $j = \overline{1, 6^4 - 1}.$
\end{enumerate}

Where $A_6$ is a permutation of columns in a vertical sector, or
permutation of vertical sectors itself, or both at the same time.

\begin{enumerate}
    \setcounter{enumi}{6}
    \item $A_7 = \{ r_s d r_t \}$ where  $r_s \neq E, \quad$ $r_t \neq E$ and  $s,t =
\overline{1, 6^4 - 1}.$
    \item $A_8 = \{ r_\lambda d r_\mu d \}$ where $r_\lambda \neq E, \quad$ $r_\mu
\neq E$ and $\lambda, \mu = \overline{1, 6^4 - 1}.$
\end{enumerate}

Where $A_8$ is a set of permutations rows and columns inside sectors
and permutations of sectors itself.

Desired group of symmetries $S = \bigcup\limits_{i = 1}^{8} A_i$. To
obtain a complete description of group $S$ we will need the
following relation.
\begin{proposition}
\label{prop1}
$$ r_{\lambda} (d r_\mu d) = (d r_\mu d) r_\lambda, \quad \forall \lambda, \forall \mu (*) $$
\end{proposition}
\begin{proof}
$$  a_{ij} (r_\lambda)\rightarrow a_{mj} (d)\rightarrow a_{jm} (r_\mu)\rightarrow a_{nm} (d)\rightarrow a_{mn}$$
$$  a_{ij} (d)\rightarrow a_{ji} (r_\mu)\rightarrow a_{ni} (d)\rightarrow a_{in} (r_\lambda)\rightarrow a_{mn}$$
\end{proof}

\newpage
The next set of relations easy to obtain from the \ref{prop1}.
\begin{corollary} $ \,$
\begin{itemize}
    \item $A_i \cap A_j = \emptyset \quad $ if $\quad i \neq j.$
    \item $A_i \cdot A_j \,$ belongs to the union $\, \forall i, \,\, \forall j \,$ sets
from $\,\{ A_i\} \,$ where $\,i = \overline{1,8}$.
\end{itemize}
( for example: $ A_7 \cdot A_8 \subset A_7 \cup A_2 \cup A_4 \cup
A_5$)
\end{corollary}

\begin{proof}

$$ a = r_s d r_t r_\lambda d r_\mu d = r_s (d r_t r_\lambda d) r_\mu d = (d r_t r_\lambda d) r_s r_\mu d = (d r^* d) r^{**} d = r^{**} d r^*$$
where $r^* = r_t r_\lambda, \quad r^{**} = r_sr_\mu.$
$$
\left\{%
\begin{array}{ll}
    \mbox{if} \quad r^* = E  \quad \mbox{and} \quad  r^{**} = E &  \quad \mbox{then}\quad  a \in A_2. \\
    \mbox{if} \quad r^* = E  \quad \mbox{and} \quad r^{**} \neq E  & \quad \mbox{then} \quad a \in A_4. \\
    \mbox{if} \quad r^* \neq E \quad \mbox{and} \quad r^{**} = E & \quad \mbox{then} \quad a \in A_5. \\
    \mbox{if} \quad r^* \neq E \quad \mbox{and} \quad r^{**} \neq E & \quad \mbox{then} \quad a \in A_7. \\
\end{array}%
\right.
$$

As a result

$$ A_7 \cdot A_8 \subset A_2 \cup A_4 \cup A_5 \cup A_7 $$

using the same approach one can show it for $\forall i, \forall j$
using the relation \ref{prop1}.
\end{proof}

It's easy to count that $$ \overline {\overline {S}} =
\sum_{i=1}^{8} \overline{\overline{A_i}};$$ where
$$ \overline{\overline{S}} = 2 + (6^4 -1) + (6^4 -1)+(6^4 -1)+(6^4 -1)+(6^4 -1)^2 + (6^4 -1)^2  = 2 \cdot 6^8. $$

If the add to the obtained group $S$ the group of interchange of the
digits (we will call this new group $O = \{o_{ij}\}; \quad i
=\overline{1, 9!}$). The group $O_i$ will be equivalent to the group
$S_9$ (9th order substitution). It's obvious that $O_i$ commutes
with any element of the group $S$ because $O_i$ commutes with $d$
and $r_i$.
$$\overline{\overline{O}} = 9!.$$

We obtain that $S \cdot O = O \cdot S$
$$ \overline{\overline{S \cdot O}} = 2 \cdot 6^8 \cdot 9!$$

The same results can be found in some other articles (for example in
the article of Royle Gordon). Let us finally show that the
constructed group $S$  exhausts all possible symmetries of the
Sudoku grid.

Permutations of the rows and columns inside the sectors,
permutations of horizontal and vertical sectors, transposition are
exhausted by sets $\{A_2\}$, $\{A_3\}$ and $\{A_6\}$. All the rest
symmetries can be obtained as a combination of some elements of the
group $S$.

\begin{enumerate}
    \item $H$ - is a symmetry with respect to the $5th$ row.
\end{enumerate}

$$ a_{ij} \rightarrow a_{10 -i, j} \quad \mbox{or} \quad H \leftrightarrows r_1$$
where
$$ r_1 = \left(%
\begin{array}{ccc}
  1\,2\,3 & 4\,5\,6 & 7\,8\,9 \\
  9\,8\,7 & 6\,5\,4 & 3\,2\,1 \\
\end{array}%
\right); \quad r_1^2 = E$$

We can identify $H = r_1, \quad r_1 \in A_3$. Elements of the
substitution $r_1$ with the numbers of the grid rows.

\begin{enumerate}
  \setcounter{enumi}{1}
    \item $H^1$ is a symmetry with respect to the $5th$ column.
$$ a_{ij} \rightarrow a_{i,10-j} \quad \mbox{or} \quad  H^1 = d r_1 d; \quad H^1 \in A_6. $$
    \item Symmetry with respect to the additional diagonal
$$ a_{ij} \rightarrow a_{10-i, 10 -j} \quad  \mbox{or} \quad D = r_1 d r_1; \quad D \in A_7. $$
    \item $V$ - rotation of rows to $\pi/2$ clockwise
$$ a_{ij} \rightarrow a_{j, 10 - i} \quad \mbox{or} \quad V = dr_1; \quad V \in A_5. $$
\end{enumerate}

Rotation of the rows defines rotation of the columns.

\begin{enumerate}
    \item $V^2$ - rotation of the rows to $\pi$.
    \item  $a_{i,j} \rightarrow a_{10-i, 10-j}$ or $V^2 = (dr_1)^2, \quad
V^2 \in A_8$.
    \item $V^3$ - rotation of the rows to $3/2 \, \pi$ clockwise.
    \item  $a_{i,j} \rightarrow a_{10-j, i}$ or $V^3 = r,d, \quad V^3 \in
A_4, \quad V^4 = E, \quad V^4 \in A_1$.
    \item  Rotation of the columns clockwise: - $W \quad$ $ W = r,d$ $W =
V^3,$ $W^2 = V^2$, $W^3 = V$, $W^4 = E$.
    \item $F$ - is a central symmetry with respect to the central element
$a_{55}$.
    \item $a_{ij} \rightarrow a_{10-i, 10-j}$ or $F = (dr_1)^2, \quad F \in
A_8$.
\end{enumerate}

It's now easy to see that $F = V^2$. Hence the group $S$ exhausts
all possible symmetries of the Sudoku grid.

There are more difficult transformations that can be defined on the
Sudoku grid. These transformations can not be reduced to the
symmetric ones but being applied they lead to the different Sudoku
grid. These more complicated transformations are related to the
questions of the unique and non-unique solvability of the particular
Sudoku grid and lots of questions are still open in this area of
research.

\end{document}